\documentclass{ws-m3as}
\usepackage{todonotes} 
\usepackage{epsfig}
\usepackage{latexsym}
\usepackage{array}
\usepackage{amssymb}
\usepackage{amsmath}
\usepackage{amsbsy}

\newcommand{\bfbeta}{{\boldsymbol\beta}}
\newcommand{\tnorm}[1]{\vert\hspace{-0.3mm}\Vert#1\Vert\hspace{-0.3mm}\vert}
\newcommand{\frakh}{\mathfrak{h}}
\newcommand{\jump}[1]{[\![#1]\!]}

\usepackage{eepic}

\begin{document}
\markboth{E. Burman}{Robust error estimates in weak norms for
  advection dominated transport}

%
\catchline{}{}{}{}{}
%

\title{Robust error estimates in weak norms for
  advection dominated transport problems with rough data}

\author{Erik Burman}

\address{Department of Mathematics,\\
University College London, Gower Street, London,\\
UK--WC1E  6BT,\\
United Kingdom\\
e.burman@ucl.ac.uk}

\maketitle

\begin{history}
\received{(Day Month Year)}
\revised{(Day Month Year)}
\comby{(xxxxxxxxxx)}
\end{history}

\begin{abstract}
We consider transient
convection--diffusion equations with a velocity vector field with
multiscale character and rough data. We assume that the velocity field
has two scales, a coarse scale with slow spatial variation, which is
responsible for advective transport and a fine scale with small
amplitude that contributes to the mixing. For this problem we consider the
estimation of filtered error quantities for solutions computed using a
finite element method with symmetric stabilization. A posteriori error
estimates and a priori error estimates are derived using the
multiscale decomposition of the advective velocity to improve stability. All estimates are independent both of the
P\'eclet number and of the regularity of the exact solution.
\end{abstract}

\keywords{passive transport; convection--diffusion ; stabilized finite
  element method ; error estimates.}

\ccode{AMS Subject Classification: 65M12, 65M60, 65M20, 65M15}

\section{Introduction}
In spite of much progress in recent years the problem of deriving a
posteriori error estimates and a priori error estimates for transient
 convection--diffusion equations is not fully understood. The difficulty is
 related to the wide range of different problems covered by the
equation class. Indeed depending on the characteristics of the
velocity field and the molecular diffusion the solutions may feature
very
different behaviour. From a computational point of view the complexity will
depend mainly on the smoothness of data and on the P\'eclet number
\[
\mathrm{Pe}_L := \frac{U L}{\mu},
\]
where $U$ is a characteristic velocity, $L$ is a lengthscale and $\mu$
the molecular diffusion.
 If the variations of the
transport velocity are small and the data are smooth, the solution
will be smooth, with moderate Sobolev norm (at least the $H^1$-norm)
independent of the viscosity. Then 
a standard Galerkin method for low P\'eclet number flows and a 
stabilized finite element method for high P\'eclet number flows will
yield accurate results.
Most difficult is the
case of a high P\'eclet number and a strongly varying, or even
turbulent, velocity field, transporting a concentration that is
strongly fragmented and may dilute or concentrate.
In this case a computation can experience strong amplification of
errors due to repeated bifurcation of streamlines and inexact
representation of internal layers caused by spurious oscillations or
numerical diffusion. Sometimes this regime is referred to as scalar turbulence and
LES-models have been derived for the modelization of the passive
scalar using filtering\cite{MCFM03}. Such models encounter a similar Reynolds
stress conundrum as the filtering of the full Navier-Stokes'
equations, and hence the modeling error is difficult to
quantify. 
Another approach that has been attempted for this problem is 
heterogeneous multiscale methods\cite{HO10} however 
in that case the underlying theory is based on homogenization and depending on a periodicity
hypothesis that in most applications will not hold.

In this paper our approach is to apply a stabilized finite
element method to the computation of the solution of the standard physical
model, instead of a coarse grained model. The accuracy of the large
scales is measured by estimating the regularized, or filtered error,
related to estimating the error in local averages of the solution.

 The combination of these
two ingredients allows us to derive error estimates with an order in
$h$, for a norm that is in a certain sense in between $H^{-1}$ and
$L^2$, but which contains the $L^2$-norm of a filtered error. These
estimates are robust in the sense that they do not depend on any high
order Sobolev norm of the exact solution and they only have 
exponential growth depending on the maximum gradient of the coarse
scale velocity field, under a certain scale separation assumption
given below. This means that the filtered quantities
considered are robust under diverging fine scale characteristic trajectories. In some
sense we extract the coarse scales for which we have some (provable) accuracy
from the computation.
\section{The transient advection--diffusion equation}
The problem that we will consider takes the following form.
Let $\Omega$ be an open, convex polygonal/polyhedral subset of $\mathbb{R}^d$,
$d=2,3$ with boundary $\partial \Omega$ and associated outward
pointing normal $n_{\partial \Omega}$.  We will denote the computational time interval by $I:=[0,T]$ and
the space-time domain by $Q:=I \times \Omega$. Assuming homogeneous Dirichlet boundary conditions
$u\vert_{\partial \Omega} = 0$ we may formally write our problem, for $t>0$ find $u \in
H^1_0(\Omega)$ such that $u(x,0) = u_0(x)$ in $\Omega$  and
\begin{equation}\label{modelconvdiff}
\partial_t u + \bfbeta \cdot \nabla u - \mu \Delta u = f, \quad
\mbox{ in } \Omega.
\end{equation}
For the problem data we consider $u_0\in L^2(\Omega)$, $ f \in L^2(Q)$
and let the velocity field
$\bfbeta \in [C_0(0,T;W^{1,\infty}(\Omega))]^d$, such
that $\nabla \cdot \bfbeta = 0$, $\bfbeta\cdot
n_{\partial \Omega}\vert_{\partial \Omega}
= 0$, and the molecular diffusivity $\mu \in \mathbb{R}$ with $\mu
> 0$. The $L^2$-scalar
product over $X$, where $X$ can be either a space domain of $\mathbb{R}^d$ or a 
space-time domain, will be denoted by $(\cdot,\cdot)_X$, the
$L^2$-scalar product over subsets $X$ of $\mathbb{R}^{d-1}$ will be
denoted $\left<\cdot,\cdot\right>_X$.
In both cases the 
corresponding $L^2$-norm is denoted by $\|\cdot\|_X$. We will use the notation $a \lesssim b$ to
denote $a \leq Cb$ with $C$ a moderate constant independent of the
mesh-parameter and the physical parameters of the problem, (except
those that are assumed to be unity). We will also use the notation $a
\sim b$ for $a \lesssim b$ and $b \lesssim a$.

In this paper the analysis will be restricted to velocity fields with
a particular multiscale character. We will assume that
the problem is normalized so that $U := \|\bfbeta\|_{L^{\infty}(Q)} =
1$ and we assume that the characteristic lengthscale is given by
$L:=1$, similarly we assume that $T \sim
L/\|\bfbeta\|_{L^{\infty}(Q)} = 1$. 
Instead of making the standard assumption that
$\|\bfbeta\|_{W^{1,\infty}(\Omega)} $ is small, we assume that there is
a decomposition of the velocity field,
\[
\bfbeta = \overline \bfbeta + \bfbeta',
\]
where, for all $t$, $\|\overline \bfbeta\|_{W^{1,\infty}(\Omega)} \sim 1$ and
$\|\bfbeta'\|_{L^\infty(\Omega)}^2 \sim \mu$. This allows us to define a 
timescale for the flow relating to the coarse scale spatial variation and
the fine scale amplitude,
\begin{equation}\label{chartime}
(\tau_F)^{-1} := \sup_{t \in I} ~ (\|\overline
\bfbeta(t)\|_{W^{1,\infty}(\Omega)}+
\|\bfbeta'(t)\|_{L^\infty(\Omega)}^{2}/\mu) \sim 1.
\end{equation}
Essentially we assume that the velocity vectorfield can be
decomposed in a coarse scale, responsible for transport, that is slowly
varying in space and a fine scale, responsible for mixing,
that has small amplitude but may have strong spatial variation.
Expressed in P\'eclet numbers this means that the coarse scale
P\'eclet number may be arbitrarily high, whereas the fine scale
P\'eclet number, based on $T$, $\bfbeta'$ and $\mu$,
must be $O(1)$. A sharper value of $\tau_F$ given a molecular
diffusion $\mu$ and a velocity field $\bfbeta$ may be obtained by solving a certain minimzation problem
in the $L^\infty$-norm that will be detailed in the a priori analysis.

We will assume that the coarse scale velocity satisfies
a pointwise non-penetration condition on the domain boundary, $\overline
\bfbeta \cdot n_{\partial \Omega} = 0$.

The rough initial data
or source term together with the high
P\'eclet number and the multiscale character of the velocity field may
lead to complex, low regularity solutions. More precisely
solutions are smooth, due to parabolic regularity, but
with large Sobolev norms, rendering standard a priori error
estimates based on approximation theory worthless. Indeed classical
global estimates for stabilized finite element methods for time dependent convection--diffusion
equations\cite{Gue01,BF09,BS11} yield the high P\'eclet number
error estimate:
\begin{equation}\label{hi_Pe_estimate}
\sup_{t \in I} \|(u - u_h)(t)\|_{\Omega} + \|\mu^{\frac12} \nabla (u - u_h)\|_Q \leq C h^{\frac32}  (1 + \mathrm{Pe}_h^{-1/2}) |u|_{L^2(I;H^2(\Omega))}
\end{equation}
where $u_h$ denotes the approximate solution using piecewise affine approximation. Even though
$|u|_{H^2(\Omega)}$ is huge in the presence of layers, stabilized
methods are relevant for this case since one may derive
localized error estimates, showing that perturbations can not spread
too far
upwind of crosswind in the stationary case\cite{JNP84,Guz06,BGL09}, or spread too far
across characteristics in the transient case\cite{Smith13}. The
reason this works is that $|u|_{H^2(\omega)}$ can be assumed to be
small in a part of the domain $\omega \subset \Omega$ provided $u$ has
no layers in a neighbourhood of $\omega$. In the
estimates the bad part can be cut away using a suitably chosen weight function.
This technique is not applicable in the present case, since the strong
oscillations of the velocity field and the nonsmoothness of $u_0$
and $f$, makes it unrealistic to assume that  $|u|_{H^2(\omega)}$ is
small in any part of the domain. The aim of this paper is to show that 
also in this case, stabilized finite element methods
produce an improved solution compared to that of standard Galerkin and
that we can indeed derive error estimates for some large scale
quantities defined by differential filtering.

Drawing from earlier ideas on a posteriori error
estimation\cite{Bu98,HMSW99} we propose to
estimate a regularized error, or in other terms, work in a norm in
between $H^{-1}$ and $L^2$. Indeed let the regularized error $\tilde
e$ be defined by the partial differential equation
\begin{equation}\label{reg_error}
-\frakh \Delta \tilde e + \tilde e = (u-u_h)(\cdot,T),
\end{equation}
with $\tilde e\vert_{\partial \Omega} = 0$, $\frakh \in \mathbb{R}^+$.
We then prove that 
\[
\|\tilde e\|_\frakh := \left(\|\frakh^{\frac12} \nabla \tilde e\|^2_{\Omega} + \|\tilde
e\|^2_{\Omega}\right)^{\frac12} \lesssim C_{T,f,u_0} \left( \frac{h}{\frakh} \right)^{\frac12}.
\]
The
constant $C_{T,f,u_0}$ of the estimates is independent both of the P\'eclet
number and of the regularity of the exact solution. It depends on the
data of the problem and the main time dependence is a factor
\[
\exp\left(c_{\Lambda} \frac{T}{\tau_F}\right)
\]
where $T$ denotes the final time of the simulation and $c_{\Lambda}$ is
  a moderate constant. This factor limits
the applicability of the analysis to time intervals of the size
$\tau_F$ and hence limits the validity of the arguments to smooth
vectorfields $\bfbeta$ or those satisfying the scale
separation case discussed above.

The parameter $\frakh$ can be related to a filter width $\delta$ by 
$\frakh:=\delta^2$, or an artificial diffusivity for elliptic
smoothing. In both cases we see that the hidden constant includes
dimensional
quantities of order unity.
For $\frakh$ constant the results can be
interpreted as $H^{-1}$-norm error estimates and in this norm the
convergence rate $h^{\frac12}$ is most likely optimal.

Below we will first recall the standard $L^2$-norm error analysis for
high P\'eclet number flow and see what is required of data to obtain
an estimate with an order that is fully independent of $\mu$ (that is,
we control suitable Sobolev norms of $u$ a priori).
Then, we consider error estimation of filtered quantities and we obtain a posteriori error
estimates.  A priori estimates for rough
solutions in the general case follow directly using standard
stability estimates of the discrete solution.


The weak formulation of equation \eqref{modelconvdiff} reads, for
$t>0$,  find $u
\in H^1_0(\Omega)$ such that $u(x,0) = u_0(x)$ and
\begin{equation}\label{weakmodel}
(\partial_t u,v)_{\Omega}  + a(u,v) = (f,v)_{\Omega} , \quad \forall v \in H^1_0(\Omega),
\end{equation}
where 
$a(\cdot,\cdot)$ is defined by:
\[
a(u,v) := (\beta \cdot \nabla u,v)_{\Omega} + (\mu \nabla u,\nabla v)_{\Omega} .
\]
\section{Stability and regularity of the solution}
Since the constant in the estimate \eqref{hi_Pe_estimate} depends on
the $H^2$-norm of the solution it is important to understand what
quantities it is possible to control, if robustness in
$\mu$ is required. In this section we will first show the standard
global regularity estimate for parabolic problems detailing the
dependence of the constant on $\mu$. We will then show
an estimate where the constant is independent of $\mu$, but not of
$\tau_F$. In both cases we will need to assume that data have some
smoothness.
\begin{lemma}\label{stand_ene}(Standard energy estimate)
Let $u$ be the solution of \eqref{weakmodel}, then there holds
\begin{equation}\label{standard_energy}
\sup_{t \in I}\|u(t)\|_{\Omega} + \|\mu^{\frac12} \nabla u\|_Q \lesssim \int_I \|f(t)\|_\Omega \mbox{d}t + \|u_0\|_\Omega.
\end{equation}
\end{lemma}
\begin{proof}
We take $v=u$ in \eqref{weakmodel} and then use that
\[
\int_I (f,u)_{\Omega} ~\mbox{d}t \leq \sup_{t \in I} \|u(\cdot,t)\|_\Omega
\int_I \|f\|_\Omega ~\mbox{d}t
\]
from which the bound on $\sup_{t \in I}\|u(t)\|_{\Omega}$
follows. This bound is then used to get the
$H^1$-bound, which is a consequence of the equation \eqref{weakmodel}
and the equality
\[
\| \mu^{\frac12} \nabla u\|_Q^2 = \int_I a(u,u) ~\mbox{d}t.
\]
\end{proof}
\begin{theorem}\label{naive_regularity}
Let $u$ be the weak solution of \eqref{modelconvdiff},
with $f \in L^2(Q)$ and $u_0 \in H^1_0(\Omega)$.
Assume that $\Omega$ is a convex domain. Then there holds
\begin{multline*}
\sup_{t\in I}\|\mu^{\frac12} \nabla u(t)\|_\Omega + |\mu u|_{L^2(I;H^2(\Omega))}
\lesssim \mu^{-1/2} \left(\int_I \|f(t)\|_\Omega
\mbox{d}t + \|u_0\|_\Omega\right) \\+  \|f\|_{Q}+\|\mu^{\frac12} \nabla u_0\|_\Omega.
\end{multline*}
\end{theorem}
\begin{proof}
Multiplying \eqref{modelconvdiff} with $-\mu \Delta u$, integrating
over the space time domain $Q^* := (0,t^*) \times \Omega$, with $t^*<T$ and applying
the Cauchy-Schwarz inequality and the arithmetic-geometric inequality gives
\begin{multline*}
\|\mu^{\frac12} \nabla u(t^*)\|^2_\Omega + \|\mu \Delta u\|^2_{Q^*} \leq
\|\bfbeta\|^2_{L^{\infty}(Q^*)} \mu^{-1} \| \mu^{\frac12} 
\nabla u\|_{Q^*}^2 + \frac14 \|\mu \Delta u\|^2_{Q^*} \\
+ \|f\|^2_{Q^*}+\|\mu^{\frac12} \nabla u_0\|^2_\Omega .
\end{multline*}
It follows from \eqref{standard_energy} that 
\[
\sup_{t \in I} \|\mu^{\frac12} \nabla u(t)\|_\Omega \lesssim
\|\bfbeta\|_{L^{\infty}(Q)} \mu^{-1/2} (\int_I \|f(t)\|_\Omega
\mbox{d}t + \|u_0\|_\Omega) +  \|f\|_{Q}+\|\mu^{\frac12} \nabla u_0\|_\Omega
\]
and
\[
\|\mu \Delta u\|_{Q} \lesssim  \|\bfbeta\|_{L^{\infty}(Q)} \mu^{-1/2} (\int_I \|f(t)\|_\Omega
\mbox{d}t + \|u_0\|_\Omega) +  \|f\|_{Q}+\|\mu^{\frac12} \nabla u_0\|_\Omega.
\]
We conclude using elliptic regularity, recalling that the advection
velocity is normalised.
\end{proof}
\begin{remark}
Observe that it follows that $\|\nabla u\|_Q \lesssim \mu^{-1/2}$, $\sup_{t \in I} \|\nabla
u\|_\Omega \lesssim \mu^{-1}$
and $\|u\|_{L^2(I;H^2(\Omega))} \lesssim
\mu^{-3/2}$. Hence we conclude that the regularity estimates of Lemma \ref{stand_ene} and Theorem
\ref{naive_regularity} both are sensitive to the variation of the
diffusivity and can only be used when $\mathrm{Pe_h} \leq
1$. Also some regularity of the inital data is needed, $u_0 \in H^1_0(\Omega)$. Revisiting
the estimate \eqref{hi_Pe_estimate} in this context we get
\begin{equation}\label{error_est_low_Pe}
\|\mu^{\frac12} \nabla (u - u_h)\|_Q \lesssim
\Bigl[\left(\frac{h}{\mu}\right)^{3/2}+ \left(\frac{h}{\mu}\right)
\Bigr]  (\int_I \|f(t)\|_\Omega
\mbox{d}t + \|u_0\|_\Omega+  \mu^{\frac12} \|f\|_{Q}+\|\mu^{\frac12} \nabla u_0\|_\Omega).
\end{equation}
This estimate is clearly useless when $\mathrm{Pe}_h > 1$.
\end{remark}

We will now show that using the scale separation of the advection
field and assuming some more regularity of data, all inverse powers of
the diffusivity may be avoided. The price to pay is an exponential
constant, which however, under our assumption on the velocity field
remains moderate. 
It should be noted that if no assumption is made on the velocity
field, sharp layers in the velocity may result in $O(\mu^{\frac12})$ width
layers in the solution $u$. The corresponding growth in the $H^1$-norm,
makes
global estimates independent of $\mu$ impossible.
\begin{theorem}\label{forward_stability}
Let $u$ be the weak solution of \eqref{modelconvdiff}, with $\Omega$ convex, $f \in
L^2(I;H^1_0(\Omega))$ and
$u_0 \in H^1_0(\Omega)$. Then for $0<\frakh$, $\mu < 1$,
there holds
\begin{multline}\label{mu_indep_est}
\sup_{t \in I} \|u(\cdot,t)\|_\frakh + |(\frakh \mu)^{\frac12} 
u|_{L^2(I;H^2(\Omega))}  + T^{-\frac12} \|\frakh^{\frac12} \nabla u\|_Q +
T^{-\frac12} \|\frakh^{\frac12} \partial_t u\|_Q
\\
\lesssim C_T
( \frakh^{\frac12} \|f\|_{L^2(I;H^1(\Omega))} + \int_I \|f\| \mbox{d}t + \|u_0\|_\frakh),
\end{multline} 
with
$
C_T = e^{\left(c_\Lambda \frac{T}{\tau_F} \right) },
$
where $\tau_F$ is given by \eqref{chartime} and $c_\Lambda$ is a moderate constant .
\end{theorem}
\begin{proof}
First observe that we know from Theorem \ref{naive_regularity} that $u
\in L^2(I;H^2(\Omega))$, so the important addition in 
Theorem \ref{forward_stability} is the control of
$\|u(\cdot,t)\|_{\frakh}$ that is robust $\mu$. Multiply equation \eqref{modelconvdiff} by $-\frakh \Delta u$ and integrate
over $Q^*:=(0,t^*)\times \Omega$ with $t^*<T$, to obtain
\begin{multline*}
\frac12 \|\frakh^{\frac12} \nabla u(\cdot,t^*)\|^2_\Omega - (\bfbeta \cdot \nabla u,
\frakh \Delta u)_{Q^*} + \|(\mu \frakh)^{\frac12} \Delta u\|^2_{Q^*}\\
\lesssim(\nabla f,\frakh \nabla u)_{Q^*}
+ \frac12 \|\frakh^{\frac12} \nabla u_0\|^2_\Omega.
\end{multline*}
The second term in the left hand side does not have a sign and
requires some further consideration. First split the velocity field in
the large and the fine scale component,
\[
- (\bfbeta \cdot \nabla u,
\frakh \Delta u)_{Q^*} =  -(\overline \bfbeta \cdot \nabla u,
\frakh \Delta u)_{Q^*}  -  (\bfbeta' \cdot \nabla u,
\frakh \Delta u)_{Q^*},
\]
 then integrate by parts in the 
term representing the large scale transport, noting that if $t_1$
and $t_2$ denotes the two orthonormal tangential vectors to $\partial \Omega$,
$$\overline \bfbeta \cdot \nabla u\vert_{\partial \Omega} = \underbrace{\overline
\bfbeta \cdot n_{\partial \Omega}}_{=0} \nabla u \cdot n_{\partial \Omega}
\vert_{\partial \Omega} + \sum_{i=1}^2 \overline
\bfbeta \cdot t_i (\underbrace{\nabla u \cdot t_i}_{=0})
\vert_{\partial \Omega}  = 0.$$
Recalling that $u \in L^2(I;H^2(\Omega))$ and $\overline \bfbeta \in
C^0(I;W^{1,\infty}(\Omega))$ we have $\overline \bfbeta \cdot \nabla u \in
L^2(I;H_0^1(\Omega)) $. Then the following integration by parts is justified
\[
-(\overline \bfbeta \cdot \nabla u,
\frakh \Delta u)_{Q^*} = (\nabla (\overline \bfbeta \cdot \nabla u),
\frakh \nabla u)_{Q^*}.
\]
Note that by the product rule
\begin{multline}\label{transport_term}
(\nabla (\overline \bfbeta \cdot \nabla u),\frakh \nabla
u)_{Q^*}  = \sum_{i=1}^d ((\partial_{x_i} \overline \bfbeta) \cdot \nabla
u, \frakh \partial_{x_i} u)_{Q^*} + \sum_{i=1}^d
(\overline \bfbeta\cdot  (\partial_{x_i}  \nabla
u), \frakh \partial_{x_i} u)_{Q^*}.
\end{multline}
For the first sum of the right hand side we have
\[
\sum_{i=1}^d ((\partial_{x_i} \overline \bfbeta) \cdot \nabla
u, \frakh \partial_{x_i} u)_{Q^*} = ((\nabla_S \overline \bfbeta) \cdot \nabla
u, \frakh \nabla  u)_{Q^*} 
\]
where $\nabla_S$ denotes the symmetric part of the gradient tensor.
Similarly we obtain for the second part
\begin{multline}
\sum_{i=1}^d
(\overline \bfbeta\cdot  (\partial_{x_i}  \nabla
u), \frakh \partial_{x_i} u)_{Q^*} = \sum_{i=1}^d \sum_{j=1}^d
(\overline \beta_j   (\partial_{x_i}  \partial_{x_j}
u), \frakh \partial_{x_i} u)_{Q^*} \\ =  \sum_{i=1}^d \sum_{j=1}^d
(\overline \beta_j   (\partial_{x_j}  \partial_{x_i}
u), \frakh \partial_{x_i} u)_{Q^*} =  \sum_{i=1}^d 
(\overline \bfbeta   \cdot \nabla \partial_{x_i}
u, \frakh \partial_{x_i} u)_{Q^*}.
\end{multline}
By the divergence theorem, recalling that $\overline \bfbeta \cdot
n_{\partial \Omega} = 0$, we have
\[
\sum_{i=1}^d 
(\overline \bfbeta   \cdot \nabla \partial_{x_i}
u, \frakh \partial_{x_i} u)_{Q^*} = - \frac12 \sum_{i=1}^d 
(\nabla \cdot \overline \bfbeta \,  \partial_{x_i}
u, \frakh \partial_{x_i} u)_{Q^*}. 
\]
We conclude that, with $\mathcal{I}$ denoting the identity matrix, 
\begin{equation}\label{transport_identity}
-(\overline \bfbeta \cdot \nabla u,
\frakh \Delta u)_{Q^*} = ((\nabla_S \overline \bfbeta - \frac12 \nabla \cdot
\overline \bfbeta \mathcal{I}) \nabla
u, \frakh \nabla u)_{Q^*} 
\end{equation}
Observing that
\begin{multline}
 (\bfbeta' \cdot \nabla u,
\frakh \Delta u)_{Q^*} \leq
\|\frakh^{\frac12}  |\bfbeta'| \mu^{-1/2} \nabla u\|_{Q^*}\|(\frakh
\mu)^{\frac12}\Delta u\|_{Q^*}\\
\leq \frac12 \|\frakh^{\frac12}  |\bfbeta'| \mu^{-1/2} \nabla u\|^2_{Q^*} + \frac12 \|(\frakh
\mu)^{\frac12}\Delta u\|_{Q^*}^2
\end{multline}
we have,
\begin{multline}
\frac12 \|\frakh^{\frac12} \nabla u(\cdot,t^*)\|^2_\Omega + \frac12 \|(\mu \frakh)^{\frac12}
\Delta u\|^2_{Q^*} \leq  \int_0^{t^*} (\nabla f , \frakh
\nabla u)_{\Omega} \\ +
\frac12\|\frakh^{\frac12} \nabla u_0\|^2_\Omega +  (\Lambda(\overline
\bfbeta,\bfbeta',\mu) \nabla u, \frakh \nabla u)_{\Omega}
\end{multline}
where 
\[
\Lambda(\overline
\bfbeta,\bfbeta',\mu) := \nabla_S \overline \bfbeta - \frac12 \nabla \cdot
\overline \bfbeta \mathcal{I} + \frac12 |\bfbeta'|^2 \mu^{-1}.
\]
Defining 
\begin{equation}\label{real_chartime}
\tilde \tau_F^{-1}  := \sup_{t \in I} \inf_{\substack{\overline\bfbeta \in W^{1,\infty}(\Omega) \\
  \overline \bfbeta \cdot n = 0 \mbox{ on } \partial \Omega}} \sigma_p(\Lambda(\overline
\bfbeta,\bfbeta',\mu))
\end{equation}
where $\sigma_p(A)$ denotes the largest positive eigenvalue of the
matrix $A$, we may write
\begin{multline}
\|\frakh^{\frac12} \nabla u(\cdot,t^*)\|^2_\Omega +  \|(\mu \frakh)^{\frac12}
\Delta u\|^2_{Q^*} \leq t^* \|\frakh^{\frac12} \nabla f \|^2_{Q^*} \\+
((t^*)^{-1}+ 2 \tilde \tau_F^{-1}) \|\frakh^{\frac12} \nabla u\|^2_{Q^*} +
\|\frakh^{\frac12} \nabla u_0\|^2_\Omega.
\end{multline}
The claim regarding the control of the space derivatives now follows
after an application of Gronwall's lemma, yielding
\begin{equation}\label{firstest}
\|\frakh^{\frac12} \nabla u(\cdot,t^*)\|^2_\Omega +
\|(\frakh \mu)^{\frac12} \Delta u\|^2_{Q^*} \lesssim e^{2 t^*/\tilde \tau_F}
\left(\| \frakh ^{\frac12}\nabla f \|_{Q^*}^2 + 
\|\frakh^{\frac12} \nabla u_0\|^2_\Omega \right).
\end{equation}
Combining this estimate with estimate \eqref{standard_energy} yields
the claim for the first two terms in the left hand side of equation \eqref{mu_indep_est}.
Note that 
we also have 
\begin{equation}\label{gradphiQ}
T^{-\frac12} \|\frakh^{\frac12} \nabla u\|^2_Q \lesssim \sup_{t \in I}
\|\frakh^{\frac12} \nabla u(\cdot,t)\|^2_\Omega \lesssim e^{2 T/\tilde
  \tau_F}
\left(\| \frakh ^{\frac12}\nabla f \|_{Q}^2 + 
\|\frakh^{\frac12} \nabla u_0\|^2_\Omega \right).
\end{equation}
For the control of the time derivative, simply multiply the equation
with $\frakh \partial_t u$ and integrate to obtain
\begin{multline*}
\frac12 \|\frakh^{\frac12} \partial_t u\|^2_Q \leq -(\bfbeta \cdot \nabla
u, \frakh \partial_t u)_Q +\frac12 \|\frakh^{\frac12}f\|^2_Q + \frac12
\|(\frakh \mu)^{\frac12}
  \nabla u_0\|^2_\Omega\\ \leq \|\bfbeta\|_{L^\infty(Q)} \|\frakh^{\frac12} \nabla
u\|_Q \|\frakh^{\frac12} \partial_t u\|_Q +\frac12 \|\frakh^{\frac12} f\|^2_Q + \frac12 \|(\frakh \mu)^{\frac12}
  \nabla u_0\|^2_\Omega
\end{multline*}
and we conclude using the estimate \eqref{gradphiQ} and the
observation that comparing the definitions \eqref{chartime} and
\eqref{real_chartime} we have $\tau_F \sim \tilde \tau_F$, with a
moderate constant.
Since $\Omega$ is convex, elliptic regularity holds
\[
|(\mu \frakh)^{\frac12} u|_{L^2(I;H^2(\Omega))} \lesssim \|(\mu
\frakh)^{\frac12} \Delta u\|_Q.
\]
The claim follows.
\end{proof}

\section{Finite element discretizations}
Let $\{\mathcal{T}_h\}$ be a family of nonoverlapping conforming,
quasi uniform
triangulations, $\mathcal{T}_h:= \{K\}_h$ where the simplices $K$ have diameter $h_K$ and that is
indexed by $h := \max h_K$. We let the set of interior faces $\{F\}_h$
of a
triangulation $\mathcal{T}_h$ be denoted by $\mathcal{F}$.

We will consider a standard finite element space of piecewise polynomial,
continuous functions
\[
V^0_h:=\{v_h \in H_0^1(\Omega): v_h\vert_K \in P_k(K),\quad \forall K \in \mathcal{T}_h \},
\]
where $P_k(K)$ denotes the polynomials of degree less than or equal
to $k$ on $K$.
The following inverse inequalities are known to hold on $V^0_h$,
\begin{equation}\label{inverse}
\|\nabla v_h\|_{K} \leq c_i h_K^{-1} \| v_h\|_{K}
\end{equation}
and
\begin{equation}\label{trace}
\|v_h\|_{\partial K} \leq c_t h_K^{-1/2} \| v_h\|_{K}.
\end{equation}
We let $\pi_h:L^2(\Omega) \rightarrow V^0_h$ denote the 
$L^2$-projection defined by $\pi_h v \in V^0_h$ such that $(\pi_h v, w_h)_\Omega = (v,w_h)_\Omega$
for all $w_h \in V^0_h$.

The standard finite element method is then obtained by restricting the
weak formulation \eqref{weakmodel} to the discrete space $V_h^0$.
For $t>0$ find $u_h \in V_h^0$ such that $u_h(x,0) = \pi_h u_0(x)$ and
\begin{equation}\label{FEMGS}
(\partial_t u_h,v_h)_\Omega + a(u_h,v_h) = (f,v_h)_\Omega, \quad \forall v_h \in V_h^0.
\end{equation}
Taking $v_h = u_h$ we immediately get the stability estimate
\[
\sup_{t \in (0,T]} \|u_h(t)\|_{\Omega} + \left( \int_0^T
\|\mu^{\frac12} \nabla u_h\|^2_{\Omega} ~\mbox{d}t \right)^{\frac12} \lesssim
\int_0^T \|f\|_{\Omega} ~\mbox{d}t + \|u_0\|_\Omega.
\]

When the P\'eclet number is low this approximations of the parabolic equation
may be analysed using well known finite element techniques, see Thom\'ee\cite{Thom97}.

However when the local P\'eclet number is high, the problem is a
singularly perturbed parabolic problem and the stability
properties of the standard Galerkin method are in general insufficient
for optimal convergence. In particular in the presence of layers the
whole computational domain may be polluted by spurious oscillations,
but as can be seen in Figure \ref{gaussian}, convergence order is lost
also for smooth solutions. In this case we study a Gaussian function convected one turn in
a disc. Time discretization is performed with the Crank-Nicolson
method and we compare the result of the standard Galerkin method with
those obtained using symmetric stabilization methods. 

In this paper we will consider symmetric stabilization methods only,
however at least for time constant $\bfbeta$ one can obtain similar
results for the SUPG-method.
\begin{figure}[h!]
\centering
\includegraphics[width=8cm]{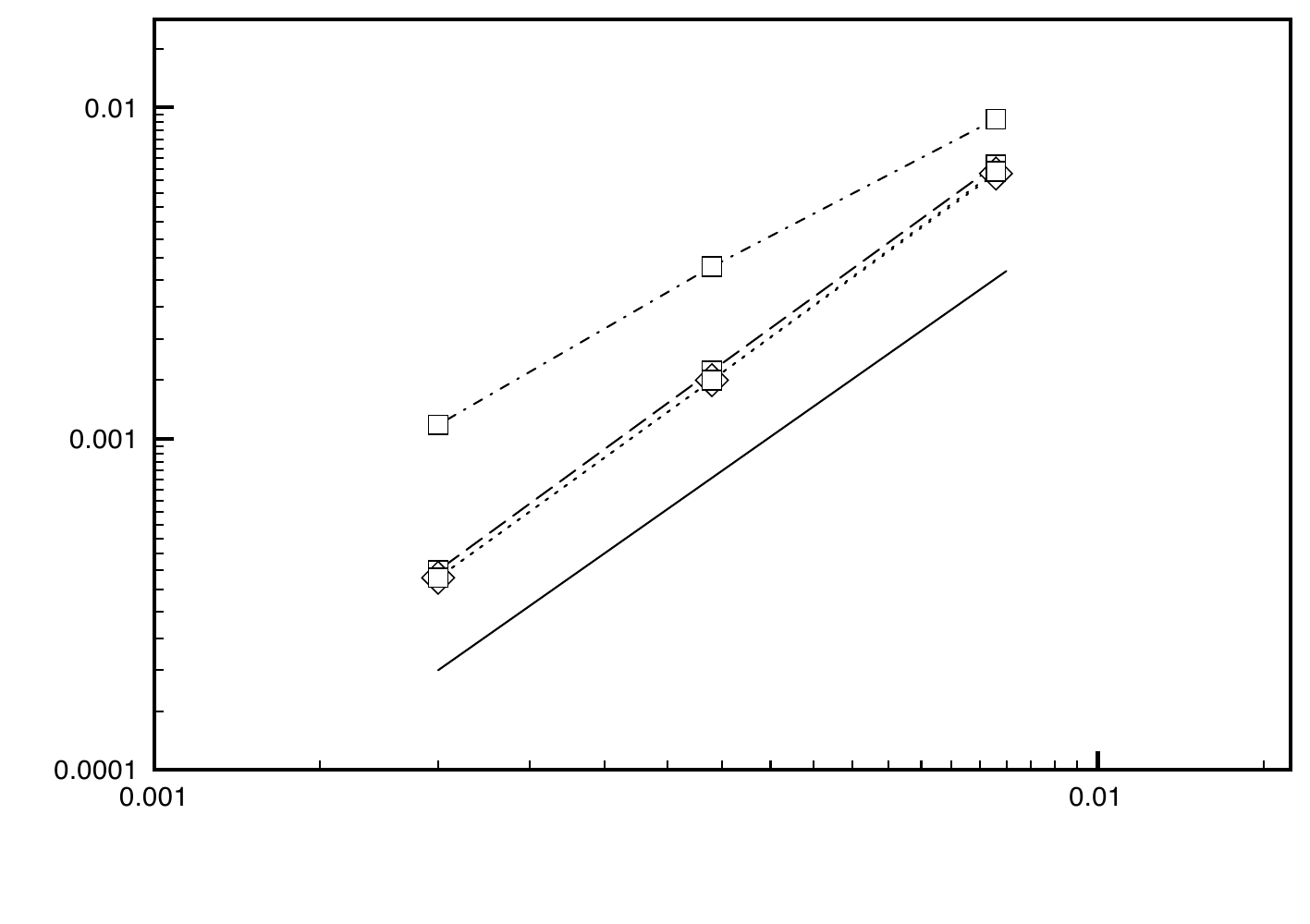}
\caption{Convergence in the $L^2$-norm plotted against the timestep $\tau=h$, of the
  stabilized (dashed line for explicit and dotted line for implicit
  treatment of the stabilization operator) and unstabilized (dash-dot) Crank-Nicolson method. The full line shows optimal second order convergence.}\label{gaussian}
\end{figure}
\subsection{Symmetric stabilization methods}
In the last ten years there has been an important development in the
field of high order symmetric stabilization methods. Such methods are
typically obtained by the addition of a weakly consistent, dissipative
operator to the formulation. 
Some of the more
important symmetric stabilization methods present in the literature
are the subgrid viscosity method suggested by Guermond\cite{Gue99}, the orthogonal
subscale method proposed by Codina\cite{Cod00}, the local project method
introduced by Becker and Braack\cite{BB04}, the discontinuous Galerkin
method\cite{JP86} and the continuous interior
penalty (CIP)  method suggested by Douglas and Dupont\cite{DD76} and analysed by Burman
and Hansbo\cite{BH04}. The analysis that we propose herein will rely on an
orthogonality argument and hence is valid only for the orthogonal
subscales, the discontinuous Galerkin method and for the CIP-method.
For simplicity we will focus on the latter below.

Herein we will assume that the following holds:
\begin{enumerate}
\item $s_h(\cdot,\cdot):V_h^0\times V_h^0 \mapsto \mathbb{R}$ is a
  symmetric, bilinear and positive semi
  definite operator.
\item For the $L^2$-projection $\pi_h:L^2(\Omega)
  \mapsto V_h^0$ there holds
\begin{itemize}
\item[--] Approximability, $\forall u \in H^1_0(\Omega)\cap H^{k+1}(\Omega)$
\begin{equation}\label{approx}
\|h^{-1/2} (u - \pi_h u)\|_{\Omega} + \|\mu^{\frac12} \nabla (u - \pi_h
  u)\|_{\Omega} \lesssim (h^{\frac12} + \mu^{\frac12}) h^{k}
  |u|_{H^{k+1}(\Omega)}, 
\end{equation}
\item[--] Stability and weak consistency of the stabilization operator
\begin{equation}
s_h(u_h,\pi_h v) \lesssim  s_h(u_h,u_h)^{\frac12} h^{\frac12} \|\nabla
v\|_\Omega,\quad \forall u_h \in V_h^0,\, v \in H^1(\Omega).
\end{equation}
\begin{equation}\label{weakconsit}
s_h(\pi_h u,\pi_h u) \lesssim h^{k+1/2}
  |u|_{H^{k+1}(\Omega)},\quad \forall u \in H^{k+1}(\Omega).
\end{equation}
\item[--]  Enhanced continuity of the convection term: for all $v_h
  \in V_h^0$ and $\bfbeta \in W^{1,\infty}(\Omega)$, with $\bfbeta \cdot
  n_{\partial \Omega} \vert_{\partial \Omega}=0$, there holds
\begin{multline}\label{enhance_cont}
|(u-\pi_h u, \bfbeta \cdot \nabla v_h)_\Omega| \lesssim
  \|\bfbeta\|_{W^{1,\infty}(\Omega)} \|u-\pi_h u\|_{\Omega}
  \|v_h\|_{\Omega} \\
+ \|\bfbeta\|^{\frac12}_{L^\infty(\Omega)} \|h^{-1/2} (u - \pi_h u)\|_{\Omega}  s_h(v_h,v_h).
\end{multline}
\end{itemize}
\end{enumerate}
As an example consider the CIP-method. Here $s_h(\cdot,\cdot)$ consists in a penalty on the jump of the gradient over
element faces, and takes the form
\[
s_h(u_h,v_h):= \gamma \sum_{F \in \mathcal{F}}\left<h^2_F
\|\bfbeta \cdot n_F\|_{L^\infty(F)} \jump{\nabla u_h \cdot n_F} ,\jump{\nabla
v_h
\cdot n_F}\right>_{F}
\]
where $F$ denotes the faces in the mesh, $\jump{x}$ the jump of $x$
over $F$, the orientation is not important, $n_F$ a fixed but arbitrary
normal associated to each face. In the
analysis below we will for simplicity use this stabilization operator.

The stabilized finite element method then takes the form, for $t>0$ find $u_h \in V_h^0$ such that $u_h(x,0) = \pi_h u_0(x)$ and
\begin{equation}\label{FEMSYMSTAB}
(\partial_t u_h,v_h)_\Omega + a(u_h,v_h) + s_h(u_h,v_h) = (f,v_h)_\Omega, \quad \forall v_h \in V_h^0.
\end{equation}
For the satisfaction of the assumptions
\eqref{approx}--\eqref{enhance_cont} we refer to Ref. \refcite{BFH06}. Although in that reference Nitsche-type boundary
conditions are used, the same argument may be shown to work whenever
$\bfbeta\cdot \nabla v_h\vert_{\partial \Omega} = 0$, which is the
case herein. For completeness we show how to obtain \eqref{enhance_cont} in our case. We also show
that the flow field $\bfbeta$ in the stabilization operator may be
replaced by $\overline \bfbeta$ at the cost of a nonessential
perturbation. The key result is the following Lemma.
\begin{lemma}\label{mod_interpol_strong}
Assume that for $k\ge 2$ no element in $\mathcal{T}_h$ has more than
one face intersecting $\partial \Omega$. Then there holds for all $u_h \in V_h^0$,
\[
\inf_{v_h \in V_h^0} \|h^{\frac12}(\pi_0 \bfbeta \cdot \nabla u_h - v_h)\|_\Omega \lesssim
s_h(u_h,u_h)^{\frac12} + h^{\frac12} \|\bfbeta\|_{W^{1,\infty}(\Omega)} \|u_h\|_\Omega,
\]
where $\pi_0 \bfbeta$ is a piecewise constant approximation of
$\bfbeta$ that will be defined below.
\end{lemma}
\begin{proof}
Let $\pi_0 \bfbeta$ be the projection onto element wise constants such
that for every $K$ that has no face on the boundary there holds
\[
\int_K \pi_0 \bfbeta ~\mbox{d}x = \int_K \bfbeta ~\mbox{d}x.
\]
On elements adjacent to the boundary define $\pi_0 \bfbeta$ by
\[
\int_{\partial K \cap \partial \Omega}  \pi_0 \bfbeta \cdot
n_{\partial \Omega} ~\mbox{d}x = \int_{\partial K \cap \partial \Omega}  \bfbeta \cdot
n_{\partial \Omega} ~\mbox{d}x
\]
and
\[
\int_{\partial K \cap \partial \Omega}  \pi_0 \bfbeta \cdot t_i
~\mbox{d}x = \int_{\partial K \cap \partial \Omega}  \bfbeta 
\cdot t_i ~\mbox{d}x,\quad  i=1,...,d-1.
\]
It follows from standard approximation that for all $K \in
\mathcal{T}_h$, $$\|\bfbeta - \pi_0 \bfbeta\|_{L^\infty(K)} \lesssim
\|\bfbeta\|_{W^{1,\infty}(K)} h_K.$$
Note that for any element $K$ with one face on the boundary there
holds 
$$
\pi_0 \bfbeta \cdot \nabla u_h\vert_{\partial \Omega \cap \partial K} = \underbrace{
\pi_0\bfbeta \cdot n_{\partial \Omega}}_{=0} \nabla u_h \cdot n_{\partial \Omega}
\vert_{\partial \Omega \cap \partial K} + \sum_{i=1}^{d} \pi_0
\bfbeta \cdot t_i (\underbrace{\nabla u_h \cdot t_i}_{=0})
\vert_{\partial \Omega \cap \partial K}  = 0.
$$
For $k=1$, if an element $K$ has two faces on the boundary, then $u_h
\equiv 0$ in that element and therefore $\pi_0 \bfbeta \cdot \nabla
u_h\vert_{\partial \Omega \cap \partial K} = 0$.
It then follows using the same arguments as in Ref. \refcite{BFH06} that
\[
\inf_{v_h \in V_h^0} \|h^{\frac12}(\pi_0 \bfbeta \cdot \nabla u_h -
v_h)\|_\Omega \leq \left(\sum_{F \in \mathcal{F}} \|h \jump{\pi_0 \bfbeta \cdot \nabla u_h}\|_F^2\right)^{\frac12}.
\]
By adding and subtracting $\bfbeta$ inside the jump and by applying a
triangle inequality followed by a trace inequality we have
\[
\left(\sum_{F \in \mathcal{F}} \|h \jump{\pi_0 \bfbeta \cdot \nabla
  u_h}\|_F^2\right)^{\frac12} \lesssim s_h(u_h,u_h)^{\frac12} + h^{\frac12}
\|\bfbeta\|_{W^{1,\infty}(\Omega)} \|u_h\|_{\Omega}
\]
and the proof is finished.
\end{proof}
The continuity \eqref{enhance_cont} is now obtained by adding and subtracting
$\pi_0 \bfbeta$ in the right slot of the left hand side of the equation
\begin{multline}\label{enhance_explain}
|(u-\pi_h u, \bfbeta \cdot \nabla v_h)_\Omega| \leq |(u-\pi_h u, (\bfbeta -
\pi_0 \bfbeta) \cdot \nabla v_h)_\Omega|+  \inf_{w_h \in V_h^0} |(u-\pi_h u, 
\pi_0 \bfbeta \cdot \nabla v_h - w_h)|_\Omega \\ \lesssim
  \|\bfbeta\|_{W^{1,\infty}(\Omega)} \|u-\pi_h u\|_{\Omega}
  \|v_h\|_{\Omega} \\
+ \|\bfbeta\|_{L^\infty(\Omega)} \|h^{-1/2} (u - \pi_h u)\|_{\Omega}  s_h(v_h,v_h).
\end{multline}
where a Cauchy-Schwarz inequality, an inverse inequality, the
approximation properties of $\pi_0 \bfbeta$ and Lemma \ref{mod_interpol_strong} have been used in the last inequality.
\begin{lemma}\label{sversuss}
Let 
\[
\bar s_h(u_h,v_h) := \gamma \sum_{F \in \mathcal{F}}\left<h^2_F
\|\overline \bfbeta \cdot n_F\|_{L^\infty(F)} \jump{\nabla u_h \cdot n_F} ,\jump{\nabla
v_h
\cdot n_F} \right>_{F}.
\]
Then
\[
 s_h(u_h,u_h) \lesssim \bar s_h(u_h,u_h) + h^{\frac12} \|\mu^{\frac12} \nabla u_h\|^2_\Omega
\]
and
\[
\bar s_h(u_h,u_h) \lesssim s_h(u_h,u_h) + h^{\frac12} \|\mu^{\frac12} \nabla u_h\|_\Omega^2.
\]
\end{lemma}
\begin{proof}
The proof for the upper and the lower bounds are similar, so we
consider only the first inequality. By using the decomposition of
$\bfbeta$ and a triangular inequality we have
\[
s_h(u_h,u_h) \lesssim \gamma \sum_{F \in \mathcal{F}} \|h_F \|\overline
\bfbeta\cdot n_F\|_{L^\infty(F)} \jump{\nabla u_h}\|_F^2 + \gamma \sum_{F \in \mathcal{F}} \|h_F \|
\bfbeta'\cdot n_F\|_{L^\infty(\Omega)} \jump{\nabla u_h}\|_F^2.
\]
Using now the size constraint on $\bfbeta'$ and a trace
inequality we conclude
\[
\gamma \sum_{F \in \mathcal{F}} \|h_F \|
\bfbeta'\cdot n_F\|_{L^\infty(\Omega)} \jump{\nabla u_h}\|_F^2 \lesssim
\gamma \sum_{K\in \mathcal{T}_h} \|h^{\frac12}_K\mu^{\frac12} \nabla u_h\|_K^2
\lesssim \gamma h^{\frac12} \|\mu^{\frac12} \nabla u_h\|^2_\Omega.
\]
\end{proof}
\subsubsection{Stability and convergence results}
We will now recall the main results on stability and convergence of
symmetric stabilization methods. These results are minor modifications
of those in Ref. \refcite{BF09}. 
For convenience we introduce a triple
norm associated to the stabilized method. Let
\[
\tnorm{u_h}^2_{h} := \int_I \left( \|\mu^{\frac12} \nabla u_h\|_\Omega^2 +
  s_h(u_h,u_h)\right) ~\mbox{d}t.
\]
Following the proof of Lemma \ref{stand_ene} it is straightforward to
derive the following stability and error estimates.
\begin{lemma}\label{stabilityCIP}(Stability)\\
Let $u_h$ be the solution of \eqref{FEMSYMSTAB}, with $\gamma\ge 0$. Then
\[
\sup_{t \in I} \|u_h(t)\|_\Omega+ \tnorm{u_h}_{h} \lesssim \int_0^T \|f\|_{\Omega} ~\mbox{d}t + \|u_0\|_\Omega.
\]
\end{lemma}
Applying stability to the error equation leads to the following error estimate:
\begin{theorem}\label{convergenceCIP}(Convergence CIP-method)\\
Let $u\in L^2(I;H^{r}(\Omega))$, $r\ge1$, be the solution of \eqref{weakmodel} and $u_h$ the
solution of \eqref{FEMSYMSTAB} {with $\gamma>0$}. Then there holds
\begin{multline*}
\sup_{t \in I} \|(u-u_h)(t)\|_\Omega+ \tnorm{u-u_h}_{h} \\ 
\lesssim 
(\tau_F^{-1} T^{\frac12} h +
(\tau_F^{-\frac12} h^{\frac12} +{ \gamma^{\frac12}+ \gamma^{-\frac12}}) h^{\frac12} + \mu^{\frac12})
h^{s-1}  |u|_{L^2(I;H^{s}(\Omega))},
\end{multline*}
with $s=\min(r,k+1)$ and where we used that $\|\bfbeta\|_{L^\infty(Q)}
\sim 1$.
\end{theorem}
\begin{proof}
For simplicity we only give the proof in the form of a final time result.
Let $e_h := u_h - \pi_h u$, using the same arguments as for the
stability Lemma \ref{stand_ene} we have
\[
\frac12 \|e_h(T)\|^2_\Omega + \tnorm{e_h}^2_h = \int_I  [(\partial_t e_h,
e_h)_{\Omega} + a(e_h,e_h) + s_h(e_h,e_h) ] ~\mbox{d}t.
\] 
Using Galerkin orthogonality and the orthogonality of the
$L^2$-projection we have
\[
\frac12 \|e_h(T)\|^2_\Omega + \tnorm{e_h}^2_h = \int_I  [ a(u-\pi_h
u,e_h)  - s_h(\pi_h u,e_h) ] ~\mbox{d}t.
\] 
Using the decomposition of the velocity we have
\[
a(u-\pi_h u,e_h) \lesssim (u - \pi_h u, \overline \bfbeta \cdot \nabla
e_h)_\Omega + (u - \pi_h u, \bfbeta' \cdot \nabla e_h) _\Omega + \|\mu^{\frac12} \nabla
(u - \pi_h u)\|_\Omega \|\mu^{\frac12} \nabla
e_h\|_\Omega
\]
and using the continuity \eqref{enhance_cont} in the first term of the
right hand side and the bound on $\bfbeta'$ in the second we have
\begin{multline*}
a(u-\pi_h u,e_h) \lesssim \|\nabla \overline \bfbeta\|_{L^\infty(\Omega)} \| (u
- \pi_h u)\|_\Omega \|e_h\|_{\Omega} \\
+ (\|\overline
\bfbeta\|^{\frac12}_{L^\infty(\Omega)} \|h^{-1/2} (u - \pi_h u)\|_\Omega  +
\tau_F^{-\frac12} \|u - \pi_h u\|_\Omega  + \|\mu^{\frac12}
\nabla (u - \pi_h u)\|_\Omega) \tnorm{e_h}_h.
\end{multline*}
It follows after a Cauchy-Schwarz inequality and an
arithmetic-geometric inequality in the right hand side that
\begin{multline*}
\|e_h(T)\|^2_\Omega + \tnorm{e_h}^2_h \lesssim  s_h(\pi_h u, \pi_h u)
+ \|\nabla \overline
\bfbeta\|^2_{L^\infty(Q)} T \|h^{-1/2} (u - \pi_h u)\|^2_Q  \\ +
(\tau_F^{-1}+ \|\bfbeta\|_{L^\infty(Q)} h^{-1}) \|u - \pi_h u\|^2_Q + \|\mu^{\frac12}
\nabla (u - \pi_h u)\|^2_Q + T^{-1} \|e_h\|^2_{Q}. 
\end{multline*}
We conclude by applying Gronwall's lemma and the approximation results
of \eqref{approx} and \eqref{weakconsit}.
\end{proof}
For high mesh P\'eclet numbers this estimate is sub optimal optimal
with $O(h^{\frac12})$ in the
$L^\infty(I;L^2(\Omega))$-norm and for low mesh P\'eclet numbers it is optimal in the
$L^2(I;H^1(\Omega))$-norm. In the latter case the convergence in the
$L^2$-norm can be improved under certain assumptions on the time
variation of $\bfbeta$\cite{Thom97}. Note that this estimate
does not have exponential growth, however in the high P\'eclet case that factor is hidden in the
Sobolev norm of the exact solution. Combining this
convergence result with the regularity result of Theorem
\ref{forward_stability} we may prove the following estimate that is
fully independent of $\mu$, in the sense that we also control the 
Sobolev norm in the constant. Note however that smoothness of the
source term and the initial data is required.
\begin{corollary}
Let $f \in L^2(I;H^1_0(\Omega))$ and $u_0 \in H^1_0(\Omega)$, let $u$
denote the solution of \eqref{weakmodel} and $u_h$ the solution of
\eqref{FEMSYMSTAB},{ with $\gamma>0$} and assume that $\mathrm{Pe}_h>>1$. Then
\[
\|(u - u_h)(T)\|_\Omega \lesssim C_T (T^{\frac12}  (1+\tau_F^{-1})+ 1) h^{\frac12} (\| f\|_{L^2(I;H^1(\Omega))} + \|\nabla u_0\|_\Omega).
\]
\end{corollary}
\begin{proof}
An immediate consequence of the Theorem \ref{convergenceCIP} and Theorem
\ref{forward_stability}.
\end{proof}
\section{Perturbation equation and the dual problem}
The error analysis in weak norms uses a perturbation equation and an associated dual
problem. 
Taking the difference of the two
formulations \eqref{weakmodel} and \eqref{FEMGS}, setting $e=u-u_h$ and integrating by parts we obtain
\[
(\partial_t e,\varphi)_\Omega+a(e,\varphi) = (e(T),\varphi(T))_\Omega -
(e(0),\varphi(0))_\Omega - (e,\partial_t \varphi + \bfbeta \cdot \nabla \varphi)_\Omega
+ (\mu \nabla e, \nabla \varphi)_\Omega.
\]
This suggests the adjoint equation, find $\varphi \in H^1_0(\Omega)$
such that
\begin{equation}\label{dualproblem}
\begin{array}{rcl}
-\partial_t \varphi - \bfbeta \cdot \nabla \varphi - \mu \Delta \varphi
&=& 0\mbox{ in } Q\\ 
\varphi  &=& 0 \mbox{ on } \partial \Omega \times I\\ 
\varphi(\cdot,T) &=& \Psi(\cdot) \mbox{ in } \Omega,
\end{array}
\end{equation}
{where $\Psi \in H^1_0(\Omega)$.}
Then the following error representation holds
\begin{equation}\label{error_rep}
(e(T),\Psi)_\Omega =  (e(0),\varphi(0))_\Omega + (\partial_t e,\varphi)_\Omega+a(e,\varphi).
\end{equation}
We will now proceed and discuss
the choice of $\Psi$ and the associated stability estimates on
$\varphi$.
\subsection{Regularization of the error and weak norms}
Since it appears not to be possible to prove a posteriori error
estimates in the $L^2$-norm independently of the P\'eclet number,
unless one ressorts to a saturation assumption, we will here consider
a regularized error, where a parameter $\frakh$ (that may ultimately depend on
$h$) sets the scale of the regularization. We recall the problem \eqref{reg_error}
for a given computational error $e:=u-u_h$,
find $\tilde e$ such that $\tilde e\vert_{\partial \Omega} = 0$ and
\[
- \frakh \Delta \tilde e + \tilde e  = e(\cdot,T).
\]
On weak form the problem writes: find
$\tilde e \in H^1_0(\Omega)$ such that
\begin{equation}\label{weakerror}
(\frakh \nabla \tilde e, \nabla v)_\Omega + (\tilde e , v)_\Omega = (e,v)_\Omega,\quad \forall v
\in H^1_0(\Omega).
\end{equation}
This is what is commonly called a Helmholtz filter or a differential
filter, although it is not properly speaking a filter. The key observation here is
that when $\frakh$ is small, the filtered solution is close to the 
solution where ever the solution is smooth. Close to layers or other strongly
localised features of $e$, $\tilde u - u$ may be large locally, also
for $\frakh$ small. Associated with the problem \eqref{weakerror} we
have the norm
\[
\|\tilde e\|^2_{\frakh}:= \|\frakh^{\frac12} \nabla \tilde e\|_\Omega^2 + \|\tilde e\|_\Omega^2
\]
and an associated relation, obtained by testing \eqref{weakerror} with
$v = \tilde e$,
\begin{equation}\label{normrep}
\|\tilde e\|^2_{\frakh} = (e,\tilde e)_\Omega.
\end{equation}
We deduce from \eqref{normrep} that choosing $\Psi = \tilde e$ in \eqref{dualproblem} above leads to
the following error representation for $\|\tilde e\|^2_{\frakh}$:
\begin{equation}\label{error_repsmooth}
\|\tilde e\|^2_{\frakh} =  (e(0),\varphi(0))_\Omega + (\partial_t e,\varphi)_\Omega+ a(e,\varphi).
\end{equation}
\subsection{Stability of the dual solution}
The advantage of using the dual technique is that instead of relying
on regularity estimates for the exact solutions we can use regularity of the
adjoint equation, which may be better behaved provided the data is
well chosen. The following Theorem gives a precise characterization of
the regularity of the dual problem in the multiscale framework.
\begin{theorem}\label{dual_stability}
Let $\varphi(x,t)$ be the weak solution of \eqref{error_rep}, with
$\Psi = \tilde e$, {where $\tilde e$ is defined by
  \eqref{weakerror}, with $u$ and $u_h$ solutions of \eqref{weakmodel} and
\eqref{FEMSYMSTAB} respectively.} Then
there holds
\[
{\sup_{t \in I} \|\varphi(\cdot,t)\|_\frakh} + T^{-1} \|\frakh^{\frac12} \nabla \varphi\|_Q +
T^{-1} \|\frakh^{\frac12} \partial_t \varphi\|_Q+
|(\frakh \mu)^{\frac12} \varphi|_{L^2(I;H^2(\Omega))} \lesssim C_T
\|\tilde e\|_\frakh,
\] 
with
$
C_T = e^{\left( c_{\Lambda} \frac{T}{\tau_F} \right) }
$
where $\tau_F$ is given by \eqref{chartime} and $c_{\Lambda}$ is a
moderate constant.
\end{theorem}
\begin{proof}
The dual problem is equivalent to the forward problem after a change
of variable $\tilde t = T - t$ and $\tilde x = -x$, with the source term
$f=0$ and the initial data $u_0 = \tilde e \in H^1_0(\Omega)$. The result then follows
from Theorem \ref{forward_stability}.
\end{proof}
\section{Error estimates}
In this section we will prove estimates for $\|\tilde
e\|_{\frakh}$ where the constant is robust in $\mu$ (under our assumptions on
the data).We will only consider the case of semi discretization in space and
show how to prove an a posteriori error estimate, where the stability
constant is essentially $C_T$ of Theorem \ref{dual_stability}. The a priori error
estimate then follows using the fact that the a posteriori residuals
are a priori controlled by the discrete stability estimate of Lemma
\ref{stabilityCIP}. Then we will
consider the case of insufficient data, i.e. when only $\overline
\bfbeta$ is known, and show that under our assumptions on data we can
obtain
an upper bound of the error also in this case, where a nonconsistent
part limits the asymptotic convergence. In the regime that we are
interested in however this part is smaller than the discretization
error.

{In the low mesh P\'eclet number regime, we need to modify our estimate to
obtain optimality. In particular we need to use elliptic regularity to
obtain optimality. We have not written the two estimates in a unified
manner, since the natural form of the residual quantities uses
different norms. We outline the differences for estimation in the low
mesh P\'eclet number regime in a remark below.}

\subsection{A posteriori and a priori error estimates} 
To prove a posteriori error estimates we use a
duality technique together with the a
priori control of the dual solution. In practice one may need to ressort to
numerical solution of the dual problem.

\begin{theorem}(A posteriori error estimate)\label{thm:apost}
Let $\tilde e$ be defined by \eqref{weakerror}, with $u$ and $u_h$ solutions of \eqref{weakmodel} and
\eqref{FEMSYMSTAB} respectively. Then there holds
\begin{multline}\label{aposteriori1}
\|\tilde e\|_\frakh \lesssim C_T \left( \frac{h}{\frakh} \right)^{\frac12}
\Bigl(\int_I \inf_{v_h \in V_h^0} \|h^{\frac12}(\bfbeta \cdot \nabla u_h -
v_h)\|_\Omega ~\mbox{d}t \\+\int_I \left(\inf_{v_h \in V_h^0} \sum_{K \in \mathcal{T}_h}
  {\|h^{\frac12}(f + \mu \Delta u_h - v_h)\|^2_K} \right)^{\frac12} \mbox{d}t+
 \int_I \left(\sum_{F \in \mathcal{F}} \|\mu  \jump{\nabla u_h\cdot n_F}\|^2_{F}
 \right)^{\frac12}\mbox{d}t\\
+ \int_I s_h(u_h,u_h)^{\frac12} \mbox{d}t + h^{\frac12} \|u_0 - \pi_h u_0\|_\Omega\Bigr).
\end{multline}
The constant $C_T$ is defined in Theorem \ref{forward_stability}.
\end{theorem}
\begin{proof}
Starting from \eqref{error_repsmooth} and using Galerkin orthogonality and the
orthogonality of the $L^2$-projection we have
\begin{equation}\label{error_rep_galortho}
\|\tilde e\|^2_{\frakh} =  (e(0),\varphi(0)-\pi_h \varphi)_\Omega +
(\partial_t e,\varphi-\pi_h \varphi)_Q+\int_I \left( a(e,\varphi-\pi_h
\varphi) +s_h(u_h,\pi_h\varphi)  \right) ~\mbox{d}t.
\end{equation}
Using the weak formulation \eqref{weakmodel} and the orthogonality of the
$L^2$-projection we may write
\begin{multline}
\|\tilde e\|^2_{\frakh} = 
(e(0),\varphi(0)-\pi_h \varphi(0))_\Omega + (\bfbeta \cdot \nabla u_h
- v_h,\pi_h \varphi- \varphi)_Q\\ + (\mu \nabla u_h,\nabla (\pi_h
\varphi -  \varphi))_Q +(f, \varphi - \pi_h \varphi)_Q +\int_I s_h(u_h,\pi_h\varphi) ~\mbox{d}t.
\end{multline}
Considering the right hand side term by term we get using
Cauchy-Schwarz inequality, approximation and the stability of Theorem \ref{dual_stability}
\begin{multline}\label{initial_cont}
(e(0),\varphi(0)-\pi_h \varphi(0))_\Omega \leq h \frakh^{-1/2}
\|e(0)\|_\Omega \sup_{t \in I} \|\frakh^{\frac12} \nabla \varphi\|_\Omega\\
\lesssim  C_T \left(\frac{h}{\frakh} \right)^{\frac12} h^{\frac12} \|e(0)\|_\Omega \|\tilde e\|_\frakh,
\end{multline}
\begin{multline}\label{conv_term}
 (\bfbeta \cdot \nabla u_h - v_h,\pi_h \varphi- \varphi)_Q 
\leq h^{\frac12} \frakh^{-1/2} \int_I \|h^{\frac12}(\bfbeta \cdot \nabla u_h -
v_h)\|_\Omega ~\mbox{d}t \sup_{t \in I}  \|\frakh^{\frac12} \nabla \varphi(\cdot,t)\|_\Omega \\
\lesssim C_T \left(\frac{h}{\frakh} \right)^{\frac12} 
\int_I \|h^{\frac12}(\bfbeta \cdot \nabla u_h -
v_h)\|_\Omega ~\mbox{d}t \|\tilde e\|_\frakh. 
\end{multline}
In the third term we first integrate by parts on each element and
then proceed with trace inequalities, followed by approximation for
the dual solution,
\begin{multline}\label{diff_cont_apost}
(f ,  \varphi - \pi_h \varphi)_Q -  (\mu \nabla u_h,\nabla (
\varphi -  \pi_h \varphi))_Q \\
\lesssim \underbrace{\int_I \left(\inf_{v_h \in V_h^0} \sum_{K\in
  \mathcal{T}_h} \|h^{\frac12}(f + \mu \Delta u_h - v_h)\|^2_K\right)^{\frac12} + \left(\sum_{F \in \mathcal{F}} \|\mu \jump{\nabla
u_h \cdot n_F}\|_F^2 \right)^{\frac12}~\mbox{d}t}_{\mathcal{R}(u_h,\mu)}\\
\times \left( \sup_{t \in I}\|h^{-1/2}(\pi_h \varphi -
  \varphi)(\cdot,t)\|_\Omega  + \sup_{t \in I} \|(\pi_h \varphi
- \varphi)(\cdot,t)\|_\mathcal{F} \right)\\
\lesssim \left(\frac{h}{\frakh}\right)^{\frac12} ~\mathcal{R}(u_h,\mu)
\sup_{t \in I} \|\frakh^{\frac12} \nabla
  \varphi(\cdot,t)\|_\Omega\\
\lesssim C_T \left(\frac{h}{\frakh}\right)^{\frac12} ~\mathcal{R}(u_h,\mu) \|\tilde e\|_\frakh.
\end{multline}
Finally the stabilization term is handled using a Cauchy-Schwarz
inequality, followed by a trace inequality and the $H^1$-stability of
the $L^2$-projection.
\begin{multline}\label{stabcont}
\int_I s_h(u_h,\pi_h\varphi) ~\mbox{d}t \lesssim 
\|\bfbeta\|^{\frac12}_{L^\infty(Q)} \int_I s(u_h,u_h)^{\frac12}
~\mbox{d}t  \,h^{\frac12} \sup_{t \in I} \|\nabla \varphi(\cdot,t)\|_\Omega \\
\\
\lesssim C_T \left(\frac{h}{\frakh}\right)^{\frac12}
\int_I s(u_h,u_h)^{\frac12}~\mbox{d}t   \|\tilde e\|_\frakh.
\end{multline}
The claim follows by collecting the upper bounds
\eqref{initial_cont}-\eqref{stabcont} and dividing by $\|\tilde e\|_{\frakh}$.
\end{proof}
\begin{theorem}(A priori error estimate)\label{thm:apriori}
{Let $\tilde e$ be defined by \eqref{weakerror}, with
  $u$ the solution of \eqref{weakmodel} and and $u_h$ the solution
\eqref{FEMSYMSTAB}, with $\gamma>0$.} Assume that $\mathrm{Pe}_h>1$ then there holds
\[
\|\tilde e\|_\frakh \lesssim C_T  \left(\frac{h}{\frakh}\right)^{\frac12} (1+h^{\frac12} + T^{\frac12}) \left(\int_0^T \|f\|_{\Omega} ~\mbox{d}t + \|u_0\|_\Omega\right).
\]
The constant $C_T$ is defined in Theorem \ref{forward_stability}.
\end{theorem}
\begin{proof}
The result follows from estimate \eqref{aposteriori1} by bounding all
the residual terms using Lemma \ref{stabilityCIP}. 
First observe that since the P\'eclet number is high we may use
inverse inequalities and trace inequalities to show that
\[
\mathcal{R}(u_h,\mu) \lesssim h^{\frac12} \int_0^T \|f\|_{\Omega}
~\mbox{d}t + T^{\frac12} \|\mu^{\frac12} \nabla
u_h\|_Q.
\]
For the contributions on the faces we have used
\begin{multline}
\int_I \|\mu  \jump{\nabla u_h \cdot n_F}\|_{\mathcal{F}} ~\mbox{d}t
 \lesssim
\|\bfbeta\|_{L^{\infty}(\Omega)}^{\frac12} h^{\frac12} T^{\frac12} \left(\int_I
  \sum_{F \in \mathcal{F}} \|\mu^{\frac12}  \jump{\nabla u_h \cdot n_F}\|^2_{F}~
 \mbox{d}t \right)^{\frac12}\\
\lesssim\|\mu^{\frac12}
 \nabla u_h\|_Q \lesssim \tnorm{u_h}_h.
\end{multline}
Similarly using a Cauchy-Schwarz inequality in time and the stability
\eqref{stabilityCIP} we have
\[
\int_I s(u_h,u_h)^{\frac12} ~\mbox{d}t \leq T^{\frac12}
 \int_I s(u_h,u_h) ~\mbox{d}t.
\]
We finally consider the first term on the right hand side of
\eqref{aposteriori1}. First note that
\[
\int_I \inf_{v_h \in V_h^0} \|h^{\frac12}(\bfbeta \cdot \nabla u_h -
v_h)\|_\Omega ~\mbox{d}t \leq
T^{\frac12}
 \inf_{v_h \in V_h^0} \|h^{\frac12}(\overline \bfbeta \cdot \nabla u_h -
v_h)\|_Q + T^{\frac12} \|h^{\frac12} \bfbeta' \cdot \nabla u_h\|_Q.
\]
By Lemma \ref{mod_interpol_strong} and Lemma \ref{sversuss} we have
\begin{multline*}
\inf_{v_h \in V_h^0} \|h^{\frac12}(\overline \bfbeta \cdot \nabla u_h -
v_h)\|_Q \lesssim  h^{\frac12} \|\overline \bfbeta\|_{W^{1,\infty}(\Omega)}
\|u_h\|_{Q} + \left(\int_I \bar s_h(u_h,u_h)~\mbox{d}t\right)^{\frac12}\\
\lesssim  h^{\frac12} \|\overline \bfbeta\|_{W^{1,\infty}(\Omega)}
\|u_h\|_{Q} + h^{\frac12} \|\mu^{\frac12} \nabla u_h\|_Q +  \left(\int_I
s_h(u_h,u_h) ~\mbox{d}t\right)^{\frac12}.
\end{multline*}
It follows that
\begin{multline}
\inf_{v_h \in V_h^0} \|h^{\frac12}(\overline \bfbeta \cdot \nabla u_h -
v_h)\|_Q \\
\lesssim \max(h^{\frac12} T^{\frac12} \|\overline
\bfbeta\|_{W^{1,\infty}(\Omega)},\|\bfbeta\|_{L^{\infty}(\Omega)}^{\frac12}) (\sup_{t \in I}
\|u_h(\cdot,t)\|_\Omega + \tnorm{u_h}_h).
\end{multline}
Using the assumption on the small scale fluctuations
$
\|\bfbeta'\|_{L^\infty(Q)}^2 \lesssim \mu 
$
we have 
\[
\|h^{\frac12} \bfbeta' \cdot \nabla u_h\|_Q \lesssim h^{\frac12} \|\mu^{\frac12}
\nabla u_h\|_Q \leq h^{\frac12} \tnorm{u_h}_h.
\]
We conclude by collecting terms and applying Lemma \ref{stabilityCIP}.
\end{proof}
{
\begin{remark}(The necessity of stabilization for robustness)
Note that the stability of the dual problem holds regardless of the
numerical method used. The stabilization in the numerical
method allows us to control the first residual in the a posteriori error
estimate, by using the discrete stability estimate
of Lemma \ref{stabilityCIP}. If no stabilization is present there is no
control of the streamline derivative making it impossible to obtain
uniformity in $\mu$. Another observation that is worthwhile is that
the above a priori error estimate is valid only for high mesh P\'eclet
number. This is because Theorem \ref{thm:apost} is optimal
only in this regime.
For low mesh P\'eclet number we may instead use the stability $|(\frakh\mu)^{\frac12}
\varphi|_{L^2(I;H^{\frac{3}{2}}(\Omega))} \leq \|\tilde e\|_{\frakh}$
in the various bounds above. We only detail how equation
\eqref{conv_term} and
\eqref{diff_cont_apost} are modified
\begin{multline}
(f - \beta \cdot \nabla u_h,  \varphi - \pi_h \varphi)_Q -  (\mu \nabla u_h,\nabla (
\varphi -  \pi_h \varphi))_Q \\
\lesssim \underbrace{\left(\int_I \inf_{v_h \in V_h^0} \sum_{K\in
  \mathcal{T}_h} \|h (f  - \beta \cdot \nabla u_h+ \mu \Delta u_h - v_h)\|^2_K+
\sum_{F \in \mathcal{F}} \|\mu h^{\frac12} \jump{\nabla
u_h \cdot n_F}\|_F^2 ~\mbox{d}t\right)^{\frac12}}_{\mathcal{R}(u_h,\mu)}\\
\times \left(\int_I \|h^{-1}(\pi_h \varphi -
\varphi)(\cdot,t)\|^2_\Omega+\|h^{-\frac12}(\pi_h \varphi
- \varphi)\|^2_\mathcal{F}  ~\mbox{d}t  \right)^{\frac12}\\
\lesssim \left(\frac{h^2 }{\frakh \mu}\right)^{\frac12} ~\mathcal{R}(u_h,\mu)
|(\mu\frakh)^{\frac12} 
  \varphi|_{L^2(I; H^2(\Omega))}\\
\lesssim C_T \left(\frac{h^2 }{\frakh \mu}\right)^{\frac12}  ~\mathcal{R}(u_h,\mu) \|\tilde e\|_\frakh.
\end{multline}
Finally
using this modified version of Theorem \ref{thm:apost}, an a priori
result with conclusion similar to that of
Theorem \ref{thm:apriori} holds for $u_h$ solution of \eqref{FEMSYMSTAB} with
$\gamma \ge 0$. The stabilization may be omitted when $\mathrm{Pe}_h <
1$.
\end{remark}}
%
\subsubsection{The degenerate case of unknown $\bfbeta'$}
In many relevant cases $\bfbeta'$ may be unknown or only partially
known. If the statistics of
$\bfbeta'$ are known some stochastic method may be used to recover
expectancy values for the solution. In this section we will consider the situation,
that $\bfbeta'$ is simply excluded from the
computation and we will show that under our assumptions on the small scale velocity
fluctuations the error estimates still hold for high mesh P\'eclet numbers. Indeed in the high
P\'eclet number regime the consistency error made by dropping the
fine scale fluctuations of $\bfbeta$ is smaller than the
discretization error.

Here we use an advective field $\overline \bfbeta$ that we assume is
divergence free and let $\bfbeta$ be replaced by $\overline \bfbeta$ in
\eqref{FEMSYMSTAB}.
\begin{theorem}\label{unknowndatathm}
Let $u$ be the solution of \eqref{weakmodel} and $u_h$ the solution of
\eqref{FEMSYMSTAB}, with $\overline \bfbeta$ instead of $\bfbeta$ for
the advective field and {$\gamma>0$}, assume that $\mathrm{Pe}_h>>1$, then
\[
\|\tilde e\|_\frakh \lesssim C_T \left( \frac{h}{\frakh}\right)^{\frac12} (1+h^{\frac12} + T^{\frac12}+T) \left(\int_0^T \|f\|_{\Omega} ~\mbox{d}t + \|u_0\|_\Omega\right).
\]
\end{theorem}
\begin{proof}
We proceed as in the proofs of Theorems \ref{thm:apost} and \ref{thm:apriori}
\begin{multline}\label{betprime_rep_galortho}
\|\tilde e\|^2_{\frakh} =  (e(0),(\varphi-\pi_h \varphi)(0))_\Omega +
(\partial_t e,\varphi-\pi_h \varphi)_Q \\+\int_I \Bigl( a(e,\varphi-\pi_h
\varphi) +\bar s_h(u_h,\pi_h\varphi) +{ (\underbrace{\overline \bfbeta \cdot \nabla
u_h - \bfbeta \cdot \nabla u_h}_{ = -\bfbeta' \cdot \nabla u_h},
\pi_h \varphi)_\Omega} dt \Bigr).
\end{multline}
The only thing that differs from the previous analysis is the last
term in the right hand side. Except for this term the proof proceeds
as before. We will therefore here only show how this term may be bounded.
After an integration by parts we have
\begin{multline*}
 (\bfbeta' \cdot \nabla u_h, \pi_h \varphi)_Q = (u_h, \bfbeta' \cdot
 \nabla \pi_h \varphi)_Q \\
\leq \sup_{t \in I}\|u_h(\cdot,t)\|_\Omega T \|\bfbeta'\|_{L^\infty(Q)}
 \frakh^{-1/2} \sup_{t \in I} \|\frakh^{\frac12} \nabla \pi_h
 \varphi(\cdot,t)\|_\Omega\\
\lesssim \sup_{t \in I}\|u_h(\cdot,t)\|_\Omega T \|\bfbeta'\|_{L^\infty(Q)}
 \frakh^{-1/2} \sup_{t \in I} \|\frakh^{\frac12} \nabla
 \varphi(\cdot,t)\|_\Omega
.
\end{multline*}
Note that under our assumptions $\|\bfbeta'\|_{L^\infty(Q)} \sim
\mu^{\frac12} \leq \|\bfbeta\|^{\frac12}_{L^\infty(Q)} h^{\frac12}$ leading
to
\[
(\bfbeta' \cdot \nabla u_h, \pi_h \varphi)_Q \lesssim \left(\frac{h}{\frakh}
\right)^{\frac12}  T C_T \|\tilde e\|_\frakh.
\]
\end{proof}
It follows that the consistency error is of the same order as the
discretization error. If the P\'eclet number is large, the
contribution from the discretization error can be assumed to be
dominating and the same order of convergence as in the unperturbed
case should be observed, until
the P\'eclet number becomes so small that the inconsistency
dominates. This means that in the high P\'eclet regime, if data are
known to be rough, noise in the velocities satisying the constraint on
$\bfbeta'$ may be neglected.
\subsection{Conclusion}
We have derived robust a posteriori and a priori error estimates for
transient convection--diffusion equations. The upshot is that the
estimates are completely robust with respect to the P\'eclet number,
in the sense that we also control the Sobolev norms of the exact
solution in the error constant. The estimates allow for low regularity
data and multiscale advection, that may have strong spatial variation
on the fine scale under a special scale separation condition. 
The aim of this work was to take a first step towards an understanding of what 
transport problems are computable in the high P\'eclet regime, beyond the standard assumption of
smooth data. These results also set a baseline for what should be achieved
theoretically in the analysis of more involved methods, such as
multiscale methods, in order to claim that they produce an accuracy beyond what
is obtained using a standard stabilized finite element method.
\section*{Acknowledgment} The author wishes to thank Professor Vivette
Girault and Professor Alexandre Ern for helpful advice.

\end{document}